\documentclass[12pt]{amsart}
\usepackage{amssymb,latexsym}
\usepackage{amsfonts}
\usepackage{amsmath}
\usepackage{algorithm}
\usepackage{algpseudocode}
\usepackage{verbatim}
\usepackage{graphicx}
\usepackage[colorlinks,linkcolor=blue,anchorcolor=blue,citecolor=blue]{hyperref}

\newcommand\C{{\mathbb C}}

\newcommand\Q{{\mathbb Q}}
\newcommand\R{{\mathbb R}}
\newcommand\Z{{\mathbb Z}}

\newcommand\N{{\mathbb N}}

\newcommand\h{\mathrm{h}}
\newcommand\al{\alpha}
\newcommand\be{\beta}
\newcommand\rem{\mathrm{rem}}

\newtheorem{theorem}{Theorem}[section]
\newtheorem{lemma}[theorem]{Lemma}
\newtheorem{conjecture}[theorem]{Conjecture}

\theoremstyle{definition}

\theoremstyle{remark}

\numberwithin{equation}{section}

\begin{document}

\title[Counting and testing dominant polynomials]{Counting and testing dominant polynomials}


\author{Art\= uras Dubickas}
\address{Department of Mathematics and Informatics, Vilnius University, Naugarduko 24,
LT-03225 Vilnius, Lithuania}
\email{arturas.dubickas@mif.vu.lt}

\author{Min Sha}
\address{School of Mathematics and Statistics, University of New South Wales,
 Sydney NSW 2052, Australia}
\email{shamin2010@gmail.com}





\begin{abstract}
In this paper, we concentrate on counting and testing dominant polynomials with integer coefficients. A polynomial is called {\em dominant} if
it has a simple root whose modulus is strictly greater than the moduli of its remaining roots. In particular, our results imply that the probability that the dominant root assumption holds for a random monic polynomial with integer coefficients tends to 1 in some setting. However, for arbitrary integer polynomials it does not tend to 1. For instance, the proportion of dominant quadratic integer polynomials of height $H$ among all quadratic integer polynomials tends to $(41+6 \log 2)/72$ as $H \to \infty$.
Finally, we will design some algorithms to test whether a given polynomial with integer coefficients is dominant or not without finding the polynomial roots.
\end{abstract}

\maketitle

\section{Introduction}\label{introduction}

Consider
$$
f(X)=a_0X^n+a_1X^{n-1}+\cdots+a_n\in \C[X]
$$
of degree $n\ge 2$. Let $\al_1,\al_2,\ldots,\al_n$ be the roots of $f$. If there exists one $\al_i$ such that $|\al_i|>|\al_j|$ for each $j \ne i$, we call $f$ \emph{dominant}, and $\al_i$ is called the \emph{dominant root} of $f$ (note that $\al_i$ must be real if $f(X) \in \R[X]$).  Dominant polynomials arise from various backgrounds (see, for instance, the motivation and the results given in \cite{Ak0,Ak1,Ak2}; one can also mention, e.g.,  linear recurrence sequences).

Recall that every linear recurrence sequence of complex numbers $s_0,s_1,s_2,\dots$ of order $n\ge 2$ is defined by the linear relation
\begin{equation}\label{sequence}
s_{k+n}=a_1s_{k+n-1}+\cdots+a_ns_k \quad (k=0,1,2,\dots),
\end{equation}
where $a_1,\dots,a_n\in \C$, $a_n\ne 0$ and $s_j \ne 0$ for at least one
 $j$ in the range $0 \le j \le n-1$.
The characteristic polynomial of this linear recurrence sequence is
$$
f(X)=X^n-a_1X^{n-1}-\cdots-a_n \in \C[X].
$$

The linear recurrence sequences with dominant characteristic polynomial, which is the so-called \textit{dominant root assumption}, are often much easier to deal with, especially, when considering Diophantine properties of linear recurrence sequences. Let us consider \textit{Pisot's conjecture} (Hadamard Quotient Theorem) as an example. Pisot's conjecture, which was proved by van der Poorten \cite{Poorten}, asserts that if the quotient $s_n/t_n$ of two linear recurrence sequences $\{s_n\}_{n\in \N}$ and $\{t_n\}_{n\in \N}$ is an integer for any $n\in \N$, then $\{s_n/t_n\}_{n\in \N}$ is also a linear recurrence sequence. Corvaja and  Zannier \cite[Theorem 1]{Corvaja1998} went further and generalized this conjecture to the case when $s_n/t_n$ is an integer infinitely often in some setting by using the subspace theorem and under the dominant root assumption. Later, in \cite[Corollary 1]{Corvaja2002} they removed the dominant root assumption.

In the first part of this paper, we consider how often the dominant root assumption holds for linear recurrence sequences. By counting dominant monic integer polynomials of fixed degree $n$ and of height bounded by $H$, we find that for fixed $n$, if in \eqref{sequence} we choose $a_1,\ldots,a_n$ as rational integers, the probability that the dominant root assumption holds tends to 1 as $H\to \infty$. Combining with \cite[Theorem 1.1]{Dubickas2014}, we see that almost every randomly generated linear recurrence sequence is non-degenerate and has a dominant root, that is, it is exactly what we usually prefer it to be.

In a similar way, we also evaluate the number of dominant
(not necessarily monic) integer polynomials of fixed degree and bounded height.

To state our results
we first define the set $S_n(H)$ of dominant monic integer polynomials of degree $n\ge 2$ and of height at most $H$, that is,
\begin{equation*}
\begin{split}
S_n(H) =  \{f(X) = &X^n +a_1X^{n-1} + \cdots + a_n\in \Z[X]~:\\
&f\ \text{is dominant}, \  |a_i| \le H, \ i =1, \ldots, n\}.
\end{split}
\end{equation*}
Similarly, we define
\begin{equation*}
\begin{split}
S_n^{*}(H) =  \{f(X) = &a_0X^n +a_1X^{n-1} + \cdots + a_n\in \Z[X]~:\\
&f\ \text{is dominant},a_0\ne 0, \  |a_i| \le H, \ i =0,1, \ldots, n\}.
\end{split}
\end{equation*}
Then, put $D_n(H)=|S_n(H)|$ and $D_n^{*}(H)=|S_n^{*}(H)|$.

Below, we shall use the Landau symbol $O$ and the Vinogradov symbol $\ll$. Recall that the assertions $U=O(V)$ and $U \ll V$ are both equivalent to the inequality $|U|\le CV$ with some constant $C>0$. In this paper, without special indication, the constants implied in the symbols $O, \ll$ only depend on the degree $n$; moreover, all these constants, except for some constants in Section \ref{complexity}, can be effectively computed. In the sequel, we always assume that $H$ is a positive integer (greater than $1$ if there is the factor $\log H$ in the corresponding formula), and $n$ is an integer greater than 1.

To determine how often dominant integer polynomials occur, we need to consider the asymptotic behavior of $D_n(H)$ and that of $D_n^{*}(H)$. First, we present a simple asymptotic formula for $D_n(H)$.

\begin{theorem}\label{Dn}
For any integer $n\ge 2$, we have
$$\lim_{H\to \infty}D_n(H)/(2H)^n=1.$$
\end{theorem}

Theorem \ref{Dn} says that the proportion of dominant monic integer polynomials of degree $n$ and of height at most $H$ among all the monic integer polynomials of degree $n$ and of height at most $H$ (there are $(2H+1)^n$ of these) tends to 1 as $H\to \infty$. Roughly speaking, the dominant monic integer polynomials occur with a probability tending to 1. Moreover, the proof of Theorem \ref{Dn} also implies an error term of this formula.

We remark that the total number of real roots of a random polynomial of degree $n$ (if the coefficients are independent standard normals) is only $\frac{2}{\pi} \log n+c$ as $n \to \infty$, where $c$ is an absolute constant (see, e.g., \cite{Ed}). That is, a random polynomial is expected to have much more non-real roots. So, Theorem \ref{Dn} is a bit surprising.

Moreover, for $0<\varepsilon\le 1$ we define the following set
\begin{equation*}
\begin{split}
S_{n,\varepsilon}(H) =  \{&f(X) = a_0X^n +a_1X^{n-1} + \cdots + a_n\in \Z[X]~:\\
&f\ \text{is dominant},0<|a_0| \le H^{1-\varepsilon}, |a_i| \le H, i =1, \ldots, n\},
\end{split}
\end{equation*}
and put $D_{n,\varepsilon}(H)=|S_{n,\varepsilon}(H)|$. Then, we can get a similar asymptotic result.

\begin{theorem}\label{Dnep}
For each $\varepsilon$ satisfying $0<\varepsilon\le 1$ and each integer $n\ge 2$, we have
$$\lim_{H\to \infty}\frac{D_{n,\varepsilon}(H)}{2H^{1-\varepsilon}(2H)^n}=1.$$
\end{theorem}

Selecting $\varepsilon=1$ in $S_{n,\varepsilon}(H)$, we obtain $a_0 \in \{-1,1\}$.
Hence,
$D_{n,1}(H)=2D_n(H)$, since half of the polynomials in $S_{n,1}(H)$ have the leading coefficient $1$ and half $-1$. Thus, Theorem~\ref{Dnep} with $\varepsilon=1$ implies Theorem~\ref{Dn}.

However, the situation for $D_n^{*}(H)$ is quite different. We can get
an explicit asymptotic formula for $D_2^{*}(H)$, but for $n \ge 3$ we can
only get lower and upper bounds.

\begin{theorem}\label{D2*}
We have
$$\lim_{H\to \infty}D_2^*(H)/(2H)^3=\frac{41+6\log 2}{72}\approx 0.6272.$$
\end{theorem}

\begin{theorem}\label{Dn*}
For any integer $n\ge 2$, we have
\begin{equation}
\liminf_{H\to \infty}D_n^*(H)/(2H)^{n+1}  \ge \frac{1}{3n^2 \sqrt{n+1}}
\notag
\end{equation}
and
\begin{equation}
\limsup_{H\to \infty}D_n^*(H)/(2H)^{n+1}\le
\left\{\begin{array}{ll}
 1-1/(3\cdot 2^{3n/2-1})&\quad\text{if $n$ is even},\\
 1-1/(3 \cdot 2^{(3n+1)/2}) &\quad\text{if $n$ is odd}.\\
\end{array}\right.
\notag
\end{equation}
\end{theorem}

It seems very likely that
the limit $\lim_{H\to \infty}D_n^*(H)/(2H)^{n+1}$
exists; see Conjecture \ref{conjecture1}. Theorem~\ref{Dn*} tells us that, contrary to the monic case, the proportion of dominant integer polynomials is positive but does not tend to 1.

After some preparations we shall prove
Theorems~\ref{Dn}-\ref{Dn*} in Section~\ref{trys}.
Then, in the second part of this paper (Section~\ref{keturi}), we apply Sturm's theorem and the Bistritz stability criterion to design algorithms on testing whether a given integer polynomial is dominant or not. By realizing these algorithms, we obtain some numerical results, which are consistent with the above theorems. Based on the numerical results, we conjecture that at least half of integer polynomials are dominant; see Conjecture \ref{conjecture2}.

\section{Preliminaries}

Given a polynomial $$f(X)=a_0X^n+a_1X^{n-1}+\cdots+a_n=a_0 (X-\al_1)\cdots (X-\al_n) \in \C[X],$$ where $a_0 \ne 0$, its {\it height}
is defined by $H(f)=\max_{0 \leq j \leq n} |a_j|$,
and its {\it Mahler measure} by
$$
M(f)=|a_0| \prod_{j=1}^n \max\{1,|\al_j|\}.
$$

For each $f(X) \in \C[X]$ of degree $n$, these quantities are related by the following well-known inequality
\begin{equation}\label{Mahler}
2^{-n}H(f)  \leq M(f) \leq \sqrt{n+1}H(f),
\end{equation}
for instance, see \cite[(3.12)]{Waldschmidt2000}.

For an algebraic number $\al \in \overline{\Q}$ of degree $d$, its Mahler measure $M(\al)$ is the Mahler measure of its minimal polynomial $f$ over $\Z$. Then, for the \emph{(Weil) absolute logarithmic height} $\h(\al)$ of $\al$, we have
\begin{equation}\label{height}
\h(\al)=\frac{\log M(\al)}{d}.
\end{equation}

Some special forms of polynomials will play an important role here.
The one below is nontrivial. It was obtained by Ferguson \cite{Ferguson1997}; see also a previous result of Boyd \cite{boyd}.

\begin{lemma}\label{Ferguson}
If $f(X) \in \Z[X]$ is an irreducible polynomial which has exactly $m$ roots on a circle $|z|=c>0$, at least one of which is real, then one has $f(X)=g(X^m)$, where the polynomial $g(X) \in \Z[X]$ has at most one real root on any circle in the plane with center at the origin.
\end{lemma}

The following lemma concerning the upper bound of the moduli of roots of polynomials is a classical result due to Cauchy \cite{Cauchy1} (see also \cite[Theorem 2.5.1 and Proposition 2.5.9]{Mignotte} or \cite[Corollary 8.3.2]{Mishra}) or \cite[Theorems 1.1.2 and 1.1.3]{Prasolov}.

\begin{lemma}\label{Cauchy}
All the roots of the polynomial of degree $n \ge 1$
$$
f(X)=a_0X^n+a_1X^{n-1}+\cdots+a_n \in \C[X],
$$
where $a_0\ne 0$ and $(a_1,\dots,a_n)\ne (0,\dots,0),$
are contained in the disc $|z|\le R$, where $X=R$ is the unique positive solution of the equation
\begin{equation*}
|a_0|X^n-|a_1|X^{n-1}-\cdots-|a_{n-1}|X-|a_n|=0.
\end{equation*}
In addition, for an arbitrary non-zero root $x$ of $f$, we have
\begin{equation}\label{Cauchy bound}
\frac{\min_{0\le i \le n}|a_i|}{H(f)+\min_{0\le i \le n}|a_i|}<|x|<1+\frac{1}{|a_0|}\max\{|a_1|,\ldots,|a_{n}|\}.
\end{equation}
\end{lemma}

This lemma will assist us in constructing a family of dominant polynomials explicitly.

For bounding the distance between two distinct roots of a complex polynomial (especially an integer polynomial), the initial work done by Mahler \cite{Mahler1964}, and then it has been studied extensively for a long time. See \cite{Budarina2013, Bugeaud2011, Bugeaud2014, Bugeaud2004, Bugeaud2010, Dubickas2013, Evertse2004} for more recent progress including some nontrivial constructions of polynomials with close roots. Usually, one needs to separate the roots of a polynomial by circles centered at these roots. However, for our purpose we also need to use separations of roots by annuli centered at the origin. So, we need to study the distance between two distinct moduli of roots of an integer polynomial.
For this,
there are two main tools that we use below.

The first one is Mahler's inequality
\cite{Mahler1964} asserting that if $\gamma$ and $\gamma'$ are two distinct roots of a separable polynomial $g(X) \in \Z[X]$ of degree $m \ge 2$ then
\begin{equation}\label{m1}
|\gamma-\gamma'|>\sqrt{3} m^{-m/2-1}M(g)^{1-m}.
\end{equation}
The following lemma is a direct consequence of \eqref{Mahler} and \eqref{m1}.

\begin{lemma}
Let $f(X)\in \Z[X]$ be a polynomial of degree $n\ge 2$, and let $\al$ and $\be$ be two distinct roots of $f$. Then, we have
\begin{equation}\label{distance0}
|\al-\be|>\sqrt{3}(n+1)^{-n-1/2}H(f)^{1-n}.
\end{equation}
\end{lemma}

The second tool is a Liouville type inequality. See, e.g.,
\cite[Lemma 3.14]{Waldschmidt2000}. We use the following version
given in \cite{fel}: if $\gamma_1,\dots,\gamma_s$ are algebraic numbers with degrees $d_1,\dots,d_s$ over $\Q$ and $P(z_1,\dots,z_s)$ is a polynomial
with integer coefficients of degree $N_1,\dots,N_s $ in the variables $z_1,\dots,z_s$, respectively, then either $P(\gamma_1,\dots,\gamma_s)=0$ or
\begin{equation}\label{m2}
|P(\gamma_1,\dots,\gamma_s)| \geq L(P)^{1-\delta d} \prod_{k=1}^s M(\gamma_k)^{-\delta N_k d/d_k},
\end{equation}
where $L(P)$ is the sum of the moduli of the coefficients of $P$,  $d=[\Q(\gamma_1,\dots,\gamma_s):\Q]$ and  $\delta=1$ (resp. $\delta=1/2$) if the field $\Q(\gamma_1,\dots,\gamma_s)$ is real (resp. complex).

We first consider quadratic integer polynomials.

\begin{lemma}\label{distance1}
Let $f(X)\in\Z[X]$ be a quadratic polynomial. Suppose that $f$ has two real roots $\al$ and $\be$ such that $|\al| \ne |\be|$. Then, we have
$$
||\al|-|\be|| \ge H(f)^{-1}.
$$
\end{lemma}
\begin{proof}
Let $f(X)=aX^2+bX+c=a(X-\al)(X-\be)$.
Since $\al$ and $\be$ are real, we have $||\al|-|\be|| =|\al - \be|$, or $|\al + \be|$. In the first case, we obtain
$$H(f)^2 |\al-\be|^2 \ge |a|^2 |\al-\be|^2=|b^2-4ac| \ge 1,$$
which implies the desired result. In the second case, $\al \ne -\be$, so
$b \ne 0$. Thus, $|\al+\be|=|-b/a| \ge 1/|a| \ge 1/H(f)$ again.
\end{proof}

Now, we consider the general case.

\begin{lemma}\label{distance}
Let $f(X)\in\Z[X]$ be a polynomial of degree $n\ge 2$, and let $\alpha$ and $\beta$ be two roots of $f$ satisfying $|\alpha| \ne |\beta|$. Then,
$$
\left||\alpha|-|\beta|\right|>
2^{n(n-1)/4}(n+1)^{-n^3/4+3n/4-3}H(f)^{-n^3/2+n^2+n/2-2}
$$
if both $\al$ and $\be$ are complex (non-real).
If, furthermore, $\al$ is real and $\be$ is complex (non-real), then
\begin{equation}\label{distance2}
\left||\alpha|-|\beta|\right| \ge 2^{-n(n-1)(n-2)/2} (n+1)^{-n(n-1)-1/2}H(f)^{-2n(n-1)-1}.
\end{equation}
Finally, if both $\al$ and $\be$ are real, then
\begin{equation}\label{distance3}
||\alpha|-|\beta||>(2n+1)^{-3n}H(f)^{2-4n}.
\end{equation}
\end{lemma}

\begin{proof}
Let us begin with the case when both $\al$ and $\be$ are real.
If $\al$ and $\be$ both have the same sign, then
$||\al|-|\be||=|\al-\be|$. If $\al$ and $\be$ have different signs,
then $||\al|-|\be||=|\al+\be|=|\al-(-\be)|$. In both cases,
$\al$ and $\be$ (or $-\be$) are the roots of the polynomial
$f(X)f(-X) \in \Z[X]$. Its separable part $g(X)$ (which is the product of the
factors of $f(X)f(-X)$
that are irreducible over $\Q$) has degree at most $2n$
and Mahler measure $M(g) \le M(f)^2$.  Clearly, $\al, \be$ and $-\be$ are the roots
of $g$. Applying Mahler's bound \eqref{m1} and
inequality \eqref{Mahler} to the polynomial $g$, we obtain
\begin{align*}
||\al|-|\be|| >\sqrt{3}(2n)^{-n-1} M(f)^{2-4n}&>(2n+1)^{-n-1}
(\sqrt{2n+1}H(f))^{2-4n}\\&=(2n+1)^{-3n}H(f)^{2-4n},
\end{align*}
as claimed.

Now, assume that $\al$ and $\be$ are both complex (non-real). Then
$n \ge 4$ and
$$2M(f)||\al|-|\be|| \ge 2\max\{|\al|,|\be|\} ||\al|-|\be||\ge ||\al|^2-|\be|^2|=|\al \bar{\al}-\be \bar{\be}|,$$
so $$||\al|-|\be|| \ge \frac{|\al \bar{\al}-\be \bar{\be}|}{2M(f)}.$$
Take a separable polynomial $f_1(X)=c_0(X-\gamma_1) \dots (X-\gamma_l) \in \Z[X]$ dividing $f(X)$ whose roots
contain $\al$ and $\be$. Observe that $\bar{\al}$ and $\bar{\be}$ are also roots of $f_1$.
Clearly, $l=\deg f_1 \le n$ and $M(f_1) \le M(f)$. Consider a separable  polynomial
$g(X) \in \Z[X]$ whose roots contain $\al \bar{\al}$ and $\be \bar{\be}$
(which is either the minimal polynomial of $\al \bar{\al}$ in $\Z[X]$ if $\be \bar{\be}$ is conjugate to $\al \bar{\al}$ or, otherwise, it is the product of the minimal polynomials of $\al \bar{\al}$ in $\Z[X]$ and that of $\be \bar{\be}$ in $\Z[X]$). It is clear that $g(X)$ divides the polynomial
$c_0^{l-1}\prod_{1 \le i < j \le l} (X-\gamma_i \gamma_j) \in \Z[X]$, so
$\deg g \le l(l-1)/2 \le n(n-1)/2$ and
\begin{align*}
M(g) &\le  |c_0|^{l-1} \prod_{1 \le i<j \le l} \max\{1,|\gamma_i \gamma_j|\} \\& \le
|c_0|^{l-1} \prod_{1 \le i<j \le l} \max\{1,|\gamma_i|\}    \max\{1,|\gamma_j|\} \\&
= |c_0|^{l-1} \prod_{k=1}^l \max\{1, |\gamma_k|\}^{l-1} = M(f_1)^{l-1} \\& \le M(f)^{n-1}.
\end{align*}
Now, as above applying Mahler's bound \eqref{m1} to the pair of roots $\al \bar{\al}, \be \bar{\be}$ of $g$ and then
inequality \eqref{Mahler} to the polynomial $g$, we obtain
\begin{align*}
||\al|-|\be|| & \ge \frac{|\al \bar{\al}-\be \bar{\be}|}{2M(f)} \\& >
\frac{\sqrt{3}}{2M(f)} \left(\frac{n(n-1)}{2}\right)^{-n(n-1)/4-1} M(f)^{(n-1)(1-n(n-1)/2)}\\& >
2^{n(n-1)/4}(n(n-1))^{-n(n-1)/4-1} M(f)^{n-2-n(n-1)^2/2}\\&>
2^{n(n-1)/4}(n+1)^{-n(n-1)/2-2}
(\sqrt{n+1}H(f))^{n-2-n(n-1)^2/2}\\
&=2^{n(n-1)/4}(n+1)^
{-n^3/4+3n/4-3}H(f)^{-n^3/2+n^2+n/2-2}.
\end{align*}

It remains to consider the case when $\al$ is real and $\be$ is complex. By Lemma \ref{distance1}, \eqref{distance2} is true when $n=2$. In the sequel, we assume that $n\ge 3$.

As above, we obtain
\begin{equation}\label{ytr}
||\al|-|\be|| \ge \frac{|\al^2-\be \bar{\be}|}{2M(f)}.
\end{equation}
In order to estimate
$|\al^2-\be \bar{\be}|$ from below, we shall apply \eqref{m2} to the polynomial $P(z_1,z_2,z_3)=z_1^2-z_2z_3$ at the point $(z_1,z_2,z_3)=(\al,\be,\bar{\be})$. Then, we have $N_1=2$, $N_2=N_3=1$, $L(P)=2$, $\delta=1/2$
and $d=[\Q(\al,\be,\bar{\be}):\Q] \le n(n-1)(n-2)$.
Also,
\begin{align*}
\frac{d}{d_1} &=\frac{[\Q(\al,\be,\bar{\be}):\Q]}{[\Q(\al):\Q]}=[\Q(\al,\be,\bar{\be}):\Q(\al)] \le [\Q(\be,\bar{\be}):\Q)] \\& \le {n(n-1)}
\end{align*}
and, similarly, $d/d_2=d/d_3 \le n(n-1)$. Thus, applying \eqref{m2}, in view
of $M(\be)=M(\bar{\be})$ we find that
$$
|\al^2-\be \bar{\be}| \ge 2^{1-n(n-1)(n-2)/2} M(\al)^{-n(n-1)}M(\be)^{-n(n-1)}.
$$
Now, from $M(\al) \le M(f)$, $M(\be) \le M(f)$ and \eqref{ytr}, we deduce
$$||\al|-|\be|| \ge 2^{-n(n-1)(n-2)/2} M(f)^{-2n(n-1)-1}.$$
By \eqref{Mahler}, $M(f)^{2n(n-1)+1} \le (n+1)^{n(n-1)+1/2}H(f)^{2n(n-1)+1} $. Hence,
$$||\al|-|\be|| \ge 2^{-n(n-1)(n-2)/2} (n+1)^{-n(n-1)-1/2}H(f)^{-2n(n-1)-1} .$$
This completes the proof of the lemma.
\end{proof}

In Lemma \ref{distance}, if $f(X)$ is irreducible and $n\ge 3$, then \eqref{distance2} can be replaced by the following
$$
\left||\alpha|-|\beta|\right| \ge 2^{-n(n-1)(n-2)/2} (n+1)^{-(n-1)(n-2)-1/2}H(f)^{-2(n-1)(n-2)-1},
$$
because $d/d_i \le (n-1)(n-2)$ for $i=1,2,3$. This will make sense for computations.

We conclude this section with the following lemma, which gives a lower bound better than \eqref{distance3} for irreducible polynomials of lower degrees (for example, of degree $n$ with $2\le n \le 11$ and of arbitrary height).

\begin{lemma}
Let $f\in\Z[X]$ be an irreducible polynomial of degree $n\ge 2$, $\alpha$ and $\beta$ two real roots of $f$.  If $|\alpha|\ne |\beta|$, then we have
$$
||\alpha|-|\beta||\ge 2^{-n(n-1)}(n+1)^{-n+1}H(f)^{-2(n-1)}.
$$
\end{lemma}
\begin{proof}
By Liouville's inequality (see \cite[(3.13)]{Waldschmidt2000}), we have
\begin{align*}
||\alpha|-|\beta||&\ge \exp\left(-[\Q(|\al|-|\be|):\Q]\h(|\al|-|\be|)\right).
\end{align*}

Note that $\al$ and $\be$ are two real algebraic numbers, we have
\begin{align*}
[\Q(|\al|-|\be|):\Q]&\le [\Q(\al,\be):\Q]\le n(n-1).
\end{align*}

In addition, combining some basic properties of the height function with \eqref{Mahler} and \eqref{height}, we have
\begin{align*}
\h(|\al|-|\be|)&\le \h(|\al|)+\h(|\be|)+\log 2\\
&\le \h(\al)+\h(\be)+\log 2\\
& =\frac{2\log M(f)}{n}+\log 2\\
& \le\frac{2}{n}\log H(f)+\frac{1}{n}\log (n+1)+\log 2.
\end{align*}

Collecting the above inequalities, we derive the desired result.
\end{proof}

\section{Counting dominant polynomials}\label{trys}

We first give several families of dominant polynomials.

\begin{lemma}\label{family1}
Let $f(X)=a_0X^n+a_1X^{n-1}+\cdots+a_n\in \Z[X]$. Suppose that $f$ is irreducible. If either $a_0>0$ and $a_i<0$ for $1\le i\le n$ or $a_0<0$ and $a_i>0$ for $1\le i\le n$, then $f$ is dominant.
\end{lemma}
\begin{proof}
Since $a_0>0$ and $a_1,\dots,a_n<0$, or $a_0<0$ and $a_1,\dots,a_n>0$, by Lemma \ref{Cauchy}, $f$ has a unique positive root $R$ such that all the other roots lie in the disc $|z|\le R$. By the choice of coefficients of $f$ and by Lemma~\ref{Ferguson}, we see that $f$ is dominant.
\end{proof}

\begin{lemma}\label{family2}
Let $f(X)=a_0X^n+a_1X^{n-1}+\cdots+a_n\in \C[X]$. If $|a_1| > n(n+1)^{1/4}|a_0|^{1/2}H(f)^{1/2}$, then $f$ is dominant.
\end{lemma}
\begin{proof}
Let $R$ be the largest modulus of the roots of $f$. Suppose that $f$ is not dominant, that is, it has at least two roots on the circle $|z|=R$. Then, by the definition of the Mahler measure, we must have $R^2\le M(f)/|a_0|$. Noticing that $|a_1|\le nR|a_0|$, and applying \eqref{Mahler}, we obtain
$$
|a_1| \le n(n+1)^{1/4}|a_0|^{1/2}H(f)^{1/2},
$$
which contradicts with our assumption.
\end{proof}

\begin{lemma}\label{family3}
Let $f(X)=(X-a)(a_0X^{n-1}+a_1X^{n-2}+\cdots+a_{n-1})\in \C[X]$ be of degree $n\ge 2$. If either $|a|\ge 2$ and $|a_0|\ge |a_i|$ for $1\le i \le n-1$ or $|a|\ge \sqrt{H(f)}+n-1$ and $|a_0|\ge 1$, then $f$ is dominant.
\end{lemma}
\begin{proof}
First, suppose that $|a|\ge 2$ and $|a_0|\ge |a_i|$ for $1\le i \le n-1$. Then, by \eqref{Cauchy bound}, the roots of $a_0X^{n-1}+a_1X^{n-2}+\cdots+a_{n-1}$ are strictly inside the circle
$$
|z|=1+\frac{1}{|a_0|}\max\{|a_1|,\ldots,|a_{n-1}|\}\le 2.
$$
So, $a$ is a dominant root of $f$, and thus $f$ is dominant.

Next, suppose that $|a_0|\ge 1$ and $|a|\ge \sqrt{H(f)}+n-1$.
Note that
$$
f(X)=a_0X^n+(a_1-aa_0)X^{n-1}+\cdots+(a_{n-1}-aa_{n-2})X-aa_{n-1}.
$$
Since $|a|> \sqrt{H(f)}$ and $H(f) \ge 1$, we obtain $|a_{i}|< \sqrt{H(f)}+n-1-i$ for $0\le i \le n-1$. By \eqref{Cauchy bound}, the moduli of the roots of $a_0X^{n-1}+a_1X^{n-2}+\cdots+a_{n-1}$ are less than
$$
1+\frac{1}{|a_0|}\max\{|a_1|,\ldots,|a_{n-1}|\}<\sqrt{H(f)}+n-1,
$$
where the inequality comes from $|a_0|\ge 1$. So, $a$ is a dominant root of $f$, and thus $f$ is a dominant polynomial.
\end{proof}

Here, we also give two families of non-dominant polynomials, which can help us to prove Theorem \ref{Dn*}.

\begin{lemma}\label{family5}
For even integer $n\ge 2$, let $f(X)=a_0X^n-a_1X^{n-1}+a_2X^{n-2}+\cdots+a_n\in \R[X]$ with $a_i>0$ for each $0\le i\le n$,
and $a_0\ge a_1, a_2\ge a_3, \cdots, a_{n-2}\ge a_{n-1}, a_n\ge a_0$. Then, $f$ is non-dominant.
\end{lemma}
\begin{proof}
Let $x$ be any positive real root of $f$. Then
$$
0=f(x)>a_0x^n-a_1x^{n-1}=x^{n-1}(a_0x-a_1),
$$
which yields $x<a_1/a_0\le 1$.

Set $g(X)=f(-X)$. Then, the negative real roots of $f(X)$ are exactly the positive real roots of $g(X)$.
Since
$$
g(X)=a_0X^n+a_1X^{n-1}+a_2X^{n-2}-a_3X^{n-3}+\cdots+a_{n-2}X^2-a_{n-1}X+a_n,
$$
it is easy to see that $g(x)>0$ for any real $x\ge 1$
(Evidently, for $n=2$ the polynomial $g(X)$ has no positive roots).
Thus, all the real roots of $f(X)$ lie in the interval $(-1,1)$.

Moreover, writing $f(X)$ as $f(X)=a_0\prod_{i=1}^{n}(X-\al_i)$, we obtain $a_n=(-1)^na_0\prod_{i=1}^{n}\al_i= a_0\prod_{i=1}^{n}\al_i$.
Since $a_n\ge a_0>0$, we see that $|\al_i|\ge 1$ for at least one index $i$. So, not all the roots of $f(X)$ lie strictly inside the unit circle. It follows that the largest
in modulus root of $f$ is complex (non-real), thus
$f(X)$ is non-dominant.
\end{proof}

\begin{lemma}\label{family6}
For odd integer $n$, let $f(X)=a_0X^n-a_1X^{n-1}+a_2X^{n-2}+\cdots+a_n\in \R[X]$ with $a_i>0$ for any
$0\le i\le n$,
and $a_0\ge a_1, a_2\ge a_3, \cdots, a_{n-1}\ge a_{n}, a_n\ge a_0$. Then, $f$ is non-dominant.
\end{lemma}
\begin{proof}
The proof is the same as that of Lemma \ref{family5}, the only one difference is that
here one should consider $g(X)=-f(-X)$.
\end{proof}

Now, by using some of the above polynomial families, we will prove the theorems stated in Section \ref{introduction}.

\begin{proof}[Proof of Theorem \ref{Dn}]
Consider the polynomial
$$f(X) = X^n +a_1X^{n-1} + \cdots + a_n\in \Z[X],$$
where $H(f)\le H$ and $|a_1| > n(n+1)^{1/4}H^{1/2}$. By Lemma \ref{family2}, we have $f\in S_n(H)$. Note that, for each sufficiently
large $H$, the number of such polynomials $f$ is at least
$$
(2H+1)^{n-1}2(H-\lfloor n(n+1)^{1/4}H^{1/2} \rfloor),
$$
which implies the desired result.
\end{proof}

\begin{proof}[Proof of Theorem \ref{Dnep}]
Evidently, $D_{n,\varepsilon}(H) \le 2H^{1-\varepsilon}(2H+1)^n$.
For the lower bound, consider the polynomials
$$f(X) = a_0X^n +a_1X^{n-1} + \cdots + a_n\in \Z[X]$$ satisfying $H(f)\le H$, $0<|a_0|\le H^{1-\varepsilon}$ and $|a_1| > n(n+1)^{1/4}|a_0|^{1/2}H^{1/2}$. Lemma \ref{family2} implies that $f\in S_{n,\varepsilon}(H)$.
Notice that the number of such polynomials $f$ is asymptotic to
$$
(2H)^{n-1}\cdot 4\sum_{a_0=1}^{\lfloor H^{1-\varepsilon} \rfloor} (H-\lfloor cH^{1/2}a_0^{1/2}\rfloor)
$$
as $H\to \infty$, where $c=n(n+1)^{1/4}$, which is asymptotic to
$$
4(2H)^{n-1}\int_{1}^{H^{1-\varepsilon}} (H-cH^{1/2}x^{1/2})\,\mathrm{d}x
$$
as $H\to \infty$. Since the above integral is equal to
$$H^{2-\varepsilon} -\frac{2}{3}cH^{2-3\varepsilon/2}+\frac{2}{3}cH^{1/2}-H,$$
 the main term is $4(2H)^{n-1}H^{2-\varepsilon}=2H^{1-\varepsilon}(2H)^n$
as $H\to \infty$. This gives the desired result.
\end{proof}

\begin{proof}[Proof of Theorem \ref{D2*}]
It is easy to see that the quadratic polynomial $f(X)=aX^2+bX+c \in \Z[X]$ is dominant if and only if $b \ne 0$ and its discriminant is greater than zero (namely, $\Delta=b^2-4ac>0$).
Since the number of polynomials $f$ with $H(f)\le H$ and $abc=0$ is $O(H^2)$, in the sequel we will assume that $abc\ne 0$.

If $ac<0$, we always have $\Delta>0$, so that $f$ is dominant. This gives exactly $2H(H^2+H^2)=4H^3$ of such dominant polynomials.
Now, suppose that $ac>0$. The two cases $a,c>0$ and $a,c<0$ yield the same number of such dominant polynomials, so we only need to consider one case, say $a,c>0$, and then multiply the result by $2$. Since $b$ can be both negative and positive,  it is obvious that the number is equal to
$$
2\sum_{a=1}^{H}\sum_{c=1}^{H}\sum_{b=\lfloor 2\sqrt{ac} \rfloor+1}^{H} 1=2\sum_{\substack{a,c=1 \\ 2\sqrt{ac} \le H}}^{H}(H-\lfloor 2\sqrt{ac}\rfloor),
$$
which is asymptotic to the following double integral
$$
\iint_{D} (2H-4\sqrt{xy})\, \mathrm{d} x \, \mathrm{d} y
$$
as $H\to \infty$, where $D=\{(x,y): 1\le x,y\le H, 2\sqrt{xy}\le H \}$.
By a direct calculation of this double integral  (here, for $1 \le x \le H/4$ we have $1 \le y \le H$ whereas for $H/4 \le x \le H$ we have $1 \le y \le H^2/(4x)$), we find that it is asymptotic to
$$
\frac{5+6\log 2}{18}H^3
$$
as $H\to \infty$.

Taking into account all the the above results, we find that
\begin{align*}
\lim_{H\to \infty}\frac{D_2^*(H)}{(2H)^3}&=\lim_{H\to \infty}\frac{4H^3+2\cdot \frac{5+6\log 2}{18}H^3}{(2H)^3}\\
&=\frac{41+6\log 2}{72},
\end{align*}
which completes the proof of the theorem.
\end{proof}

\begin{proof}[Proof of Theorem \ref{Dn*}]
Consider the polynomial $$f(X) = a_0X^n +a_1X^{n-1} + \cdots + a_n\in \Z[X],$$
where $H(f)\le H$ and $|a_1| > n(n+1)^{1/4}|a_0|^{1/2}H^{1/2}$. By Lemma \ref{family2}, we have $f\in S_n^*(H)$.
Set $c=n(n+1)^{1/4}$.
The number of such polynomials $f$ is asymptotic to
$$
(2H)^{n-1}\cdot 4\sum_{a_0=1}^{\lfloor H/c^2 \rfloor} (H-\lfloor cH^{1/2}a_0^{1/2} \rfloor)
$$
as $H\to \infty$, which is asymptotic to the following integral
$$
4(2H)^{n-1}\int_{1}^{ H/c^2 } (H-cH^{1/2}x^{1/2})\,\mathrm{d}x
$$
as $H\to \infty$. Note that the main term in the above integral is $H^2/(3c^2)$
as $H \to \infty$. The main term for the number of such polynomials $f$ is thus
$(2H)^{n+1}/(3c^2)$, which gives the desired lower bound of the lower limit.

Now, we want to derive the claimed upper bound for the upper limit. If $n$ is even,
we will count the polynomials $f(X)$ with
integer coefficients as in
Lemma \ref{family5}. Since $a_n \ge a_0 \ge a_1$ and $a_{2i} \ge a_{2i+1}$
for each $i=1,\dots,n/2-1$, it is easy to find that the number of these polynomials $f(X)$ is asymptotic to
$$
\frac{H^3}{6} \cdot \left(\frac{H^2}{2}\right)^{n/2-1} = \frac{H^{n+1}}{3\cdot 2^{n/2}}
$$
as $H\to \infty$.
Notice that $f(-X)$, $-f(X)$ and $-f(-X)$ are also non-dominant, and they are different polynomials.
So, as $H \to \infty$,  we get $H^{n+1}/(3 \cdot 2^{n/2-2})$ non-dominant polynomials. Thus, we have
$$
\limsup_{H\to \infty}D_n^*(H)/(2H)^{n+1}\le 1-\frac{1}{3\cdot 2^{3n/2-1}},
$$
where $n$ is even.

If $n$ is odd, similarly as the above, the number of polynomials $f(X)$ with integer coefficients as
in Lemma \ref{family6} (so satisfying  $a_{n-1} \ge a_n \ge a_0\ge a_1$ and $a_{2i}\ge a_{2i+1}$ for $i=1,\dots, (n-3)/2$) is asymptotic to

$$
\frac{H^4}{24} \cdot \left(\frac{H^2}{2}\right)^{(n-3)/2} = \frac{H^{n+1}}{3\cdot 2^{(n+3)/2}}
$$
as $H\to \infty$.
Now, as $H \to \infty$, multiplying by $4$ as above, we also get $H^{n+1}/(3 \cdot 2^{(n-1)/2})$ of non-dominant polynomials. Thus, for $n$ odd, we obtain
$$
\limsup_{H\to \infty}D_n^*(H)/(2H)^{n+1}\le 1-\frac{1}{3 \cdot 2^{(3n+1)/2}},
$$
as claimed.
\end{proof}

Consider cubic polynomials $f(X)=a_0X^3+a_1X^2+a_2X+a_3\in \Z[X]$ with $H(f)\le H$. The number of these polynomials $f$ with discriminant zero or $a_1a_2a_3$ zero is $O(H^3)$, and the number of such reducible polynomials $f$ is also $O(H^3)$. It is well-known that $f$ has three distinct real roots if its discriminant $\Delta$ is positive, and that it has two conjugate complex roots if $\Delta<0$. So, in particular, $f$ is dominant if $\Delta>0$. If $f$ is irreducible, $\Delta<0$ and $a_1a_2a_3 \ne 0$, then not all three
roots of $f$ lie on the same circle in view of Lemma~\ref{Ferguson};
hence, either $f(X)$ or its reciprocal polynomial $X^3f(X^{-1})$ is dominant. Thus, at least half of all cubic integer polynomials $f$ are dominant.

\section{Testing dominant polynomials}\label{keturi}

In this section, we will design some algorithms to test whether a given polynomial $f\in \Z[X]$ is dominant or not without finding the polynomial roots.
We first recall Sturm's theorem and the Bistritz stability criterion.

\subsection{Sturm's theorem}

For an arbitrary real sequence $a_0,a_1,\cdots,a_n$, its \emph{number of sign changes}
is determined as follows:
one first deletes all the zero terms of the sequence, then for the remaining non-zero terms one counts the number of
pairs of neighboring terms of different sign.

Given a polynomial $f(X)\in \R[X]$, applying Euclid's algorithm to $f(X)$ and its derivative
yields the following construction:
\begin{align*}
& p_0(X):= f(X),\\
&p_1(X):= f^\prime(X),\\
&p_2(X):= -\rem(p_0(X),p_1(X))=p_1(X)q_0(X)-p_0(X),\\
&p_3(X):= -\rem(p_1(X),p_2(X))=p_2(X)q_1(X)-p_1(X),\\
&\dots \\
&0=-\rem(p_{m-1}(X),p_{m}(X)), \quad \textrm{where $p_{m}(X)\ne 0$},
\end{align*}
where $\rem (p_i(X),p_j(X))$ and $q_i(X)$ are the remainder and the quotient of the polynomial division of $p_i(X)$ by $p_j(X)$, respectively.

Now we state Sturm's theorem as follows; see \cite[Theorem 1.4.3]{Prasolov}.
\begin{theorem}\label{thm:Sturm}
Let $f(X),p_0(X),p_1(X), \cdots, p_m(X)$ be as above.
For any $x\in \R$, let $\sigma(x)$ be the number of sign changes of the
sequence
$$p_0(x),p_1(x), \cdots, p_m(x).$$
Given real numbers $a,b$ with $a<b$, suppose that $f(a)f(b)\ne 0$, then the number of distinct real roots of $f(X)$
in the interval $(a,b)$ is equal to $\sigma(a)-\sigma(b)$.
\end{theorem}

\subsection{Bistritz stability criterion}\label{Bistritz}

We say that a real polynomial is \emph{stable} if all of its roots lie strictly inside the unit circle.
The Bistritz stability criterion, arising from signal processing and control theory and developped by Y. Bistritz, is a simple method
to determine whether a given real polynomial is stable or not;
see \cite{Bistritz1984} for more details and also \cite{Bistritz1986,Bistritz2002}.

For any real polynomial $f(X)\in \R[X]$ of degree $n\ge 1$, we denote by $f^*(X)$ the \emph{reciprocal polynomial} of
$f(X)$, that is $f^*(X)=X^nf(1/X)$.

Consider
$$
A_n(X)=a_0X^n+\cdots+a_n\in \R[X],
$$
where $n\ge 2$, $a_0\ne 0$ and $A_n(1)\ne 0$ (If  $A_n(1)=0$,  the polynomial is not stable.).
We assign to $A_n(X)$ a sequence of symmetric polynomials
$$T_m(X)=T_m^*(X),    m=n,n-1, \ldots , 0,$$
created by a three-term polynomial recursion as follows.

Initiation:
\begin{equation}\label{initiation}
 T_n(X)=A_n(X)+A_{n}^*(X), \quad  T_{n-1}(X)=\frac{ A_n(X)-A_{n}^*(X) }{X-1}.
\end{equation}

Recursion: for $k=n, n-1,\ldots,2$, define
\begin{equation}\label{recursion}
 T_{k-2}(X)=\frac{\delta_{k}(X+1) T_{k-1}(X) - T_{k}(X) }{X},
\end{equation}
 where  $\delta_{k}=\frac{T_{k}(0)}{T_{k-1}(0)}$. This recursion requires the \emph{normal conditions}:
 $$
T_{k-1}(0)\ne 0, \quad k=n,\ldots, 1,
 $$
 which mean that each polynomial $T_{k-1}(X)$ is of degree $k-1$. The recursion is interrupted when a $T_{k-1}(0)=0$ occurs ($k\ge 2$),
 which corresponds to a \emph{singular case}.

 We classify the singular cases into the following two classes:
the case when $T_{k-1}(X)= 0$ is called the first type singularity, and the case when $T_{k-1}(0)= 0$ but $T_{k-1}(X)\ne 0$
 is called the second type singularity.

There is a necessary and sufficient condition for the occurrence of the first type singularity.

\begin{theorem}\label{thm:singular1}
Let $A_n(X)\in \R[X]$ be of degree $n\ge 2$. Then, in the normal conditions
for $A_n(X)$,
the first type singularity occurs if and only if $A_n(X)$ has roots on the unit circle or reciprocal pair roots $(x,x^{-1})$.
\end{theorem}

Based on Theorem \ref{thm:singular1}, we will see later that for our purpose, essentially we don't need to afraid the first type singularity, see Lemma \ref{singular1}.

However, the appearance of the second type singularity does not correspond to a special pattern of roots location, except that it implies an unstable polynomial.
So, here we need to handle the second type singularity.
Starting from $n$ to 1, if $k$ is the first integer for which $T_{k-1}(0)= 0$ and $T_{k-1}(X)\ne 0$, we replace $T_k(X)$ and $T_{k-1}(x)$ by
$$
T_k(X)+(X-1)T_{k-1}(X)(X^q-X^{-q})
$$
and
$$
T_{k-1}(X)(K+X^q+X^{-q}),\quad K>2,
$$
respectively, where $q$ is the multiplicity of $T_{k-1}(X)$ at $X=0$ and $K$ is an arbitrary real constant greater than 2; after this replacement we can see that $T_{k-1}(0)\ne 0$. Then, we return to the recursion \eqref{recursion}. As soon as we encounter the second type singularity, we apply the same treatment again. Finally, we will also obtain the sequence
$T_n(X), T_{n-1}(X), \ldots, T_0(X)$.

\begin{theorem}\label{thm:Bistritz}
Let $A_n(X), T_0(X), \ldots, T_n(X)$ be as above (including the situation that the second type singularity occurs). Define $\nu_n$ as the number of sign changes of the
sequence
$$
T_n(1),\ldots,T_0(1).
$$
Then, $A_n(X)$ is stable if and only if the normal conditions hold and $\nu_n=0$.
Furthermore, if the normal conditions hold or only the second type singularity  occurs, then $A_n(X)$ has $\nu_n$ roots strictly outside the unit circle, and $n-\nu_n$ roots strictly inside the unit circle (counted with multiplicity).
\end{theorem}

\subsection{Algorithms for testing dominant polynomials}\label{algorithm}

For a quadratic real polynomial $f=aX^2+bX+c\in \R[X]$, the testing problem is easy, because $f$ is dominant if and only if $b\ne 0$ and $b^2-4ac>0$.
Now, we will apply Sturm's theorem and the Bistritz stability criterion to design general algorithms for testing dominant integer polynomials.

In fact, there are several ways to design such a general algorithm.
We first present a simple algorithm (see Algorithm \ref{test1}) as a comparison, and then provide a more efficient one in practice (see Algorithm \ref{test2}). Finally, we will give a slight improvement of Algorithm \ref{test2} in the case of irreducible polynomials; see Algorithm \ref{test3}.

\subsubsection{General case I}

Following \eqref{distance0} and \eqref{distance2}, for any polynomial $f\in \Z[X]$ of degree $n\ge 2$, we define the following:
$$
d_1(f)= 2^{-n(n-1)(n-2)/2} (n+1)^{-n(n-1)-1/2}H(f)^{-2n(n-1)-1} ,
$$
and
$$
d_2(f)=\sqrt{3}(n+1)^{-n-1/2}H(f)^{1-n}.
$$

Moreover, for $f(X)=a_0X^n+a_1X^{n-1}+\dots+a_n\in \Z[X]$, we define
$$
C_1(f)=\frac{1}{1+H(f)}
$$
and
$$
C_2(f)=1+\frac{1}{|a_0|}\max\{|a_1|,\ldots,|a_{n}|\}.
$$
Then, by \eqref{Cauchy bound}, every root of $f$ is of modulus greater than $C_1(f)$ and less than $C_2(f)$.

Given a polynomial $f(X)\in\Z[X]$ and a real root $x$ of $f$, if $x$ is located in a small annulus $r< |z|< R$
with $R-r<d_1(f)$, then, by \eqref{distance2}, we can see that
there are no roots of $f$ located in this annulus and with
modulus not equal to $|x|$.

\begin{lemma}\label{singular1}
Let $f(X)\in \Z[X]$ be of degree $n\ge 2$, and assume that there is a real root $x$ of $f$ lying in the annulus $r<|z|<R$,
where $R-r\le\frac{1}{2}d_1(f)$.
We define $g(X)=f(rX)$, and consider the normal conditions of $g(X)$.
 If the first type singularity occurs, then $f$ has a root $y$ such that $|y|>|x|$.
\end{lemma}
\begin{proof}
Suppose that the first type singularity occurs for $g(X)$. By Theorem \ref{thm:singular1}, we know that
$g(X)$ has roots on the unit circle or a reciprocal pair of roots.

If $g(X)$ has a root $\al$ on the unit circle, then $r\al$ is a root of $f$ and
$|r\al|=r$. But by the choice of
$r$ and $R$, $f$ has no roots on the circle $|z|=r$. This is a contradiction. So, $g(X)$ has a reciprocal pair of roots.

In the sequel, we let $(\al,\al^{-1})$ be this reciprocal pair roots of $g(X)$ with $|\al|<1$. So,
$f$ has roots $r\al$ and $r\al^{-1}$.

Assume that $|r\al^{-1}|=|x|$. Since $|r\al|<r<|x|$, we have
\begin{equation}\label{rz1}
0<|x|-|r\al|=|x|-\frac{r^2}{|x|}<R-\frac{r^2}{R}\le (1+\frac{r}{R})\cdot \frac{1}{2}d_1(f)<d_1(f).
\end{equation}
Notice that both $x$ and $r\al$ are roots of $f$ and $x$ is real. Therefore, \eqref{rz1} contradicts \eqref{distance2}.

Now, assume that $|r\al^{-1}|<|x|$. Since $|\al^{-1}|>1$, we have
\begin{equation}
0<|x|-|r\al^{-1}|<|x|-r<R-r\le \frac{1}{2}d_1(f).
\notag
\end{equation}
As the above, this also yields a contradiction.

Thus, we must have $|r\al^{-1}|>|x|$. So, $r\al^{-1}$ is exactly a root we need,
this completes the proof.
\end{proof}

Now we will explain Algorithm \ref{test1} step by step.

Step \ref{state1-1}:
Let $R_+$ be the largest positive root of $f$ (if it exists), and let
$R_-$ be the largest modulus of the negative roots of $f$ (if it exists).
Suppose that $R_+$ exists.
Applying Algorithm \ref{alg:Sturm} to $f$ in the interval $(C_1(f),  C_2(f))$ by setting $d=\frac{1}{2}d_1(f)$, we locate the positive root $R_+$ in an annulus $r_1< |z| < R_1$ such that $R_1-r_1\le \frac{1}{2}d_1(f)$. Then, we apply a similar algorithm to $f$ in the interval $(-C_2(f), -R_1)$. Note that we might have $R_1=C_2(f)$, but this case also can be excluded by Sturm's theorem. If there are no roots locating in this interval, then return $r=r_1, R=R_1$; otherwise,  we can get an annulus $r_2< |z| < R_2$
such that $r_2< R_- < R_2$ and $R_2-r_2\le \frac{1}{2}d_1(f)$, then return $r=r_2, R=R_2$. If $R_+$ doesn't exist, we search negative roots of $f$ in the interval $(-C_2(f), -C_1(f))$.

Step \ref{state1-2}:
Since $g(X)$ already has a real root strictly outside the unit circle, it is not stable. So, the normal conditions of
$g(X)$ may not hold. When the first type singularity occurs, by Lemma \ref{singular1}, we see that
$f$ is not dominant. If the normal conditions hold or only the second type singularity occurs, following the discussions in
Section \ref{Bistritz} we construct the sequence $T_n(X),\ldots, T_0(X)$ and compute the number of sign changes $\nu_n$ of the sequence $T_n(1),\ldots, T_0(1)$. Then, by Theorem \ref{thm:Bistritz}, $g(X)$ has $\nu_n$ roots strictly outside the unit circle.

Step \ref{state1-3}:
If $g(X)$ has only one root strictly outside the unit circle ($\nu_n=1$), then $f$ has only one root strictly
outside the circle $|z|=r$. Since we already know that there exists one real root
located in the annulus $r< |z| < R$, we can deduce that $f$ is dominant.
Otherwise, $f$ is not dominant.

\begin{algorithm}
\caption{Simple test of dominant polynomials}
\label{test1}
\begin{algorithmic}[1]
\Require polynomial $f(X)\in \Z[X]$ of degree $n\ge 2$.
\Ensure 1 ($f$ is dominant) or 0 ($f$ is not dominant).
\State Use Sturm's theorem to locate a real root with the largest modulus among the real roots of $f$ in a small annulus $r< |z| < R$ such that $R-r\le \frac{1}{2}d_1(f)$. If no non-zero real roots exist, return 0. \label{state1-1}
\State Let $g(X)=f(rX)$, and apply the Bistritz stability criterion to calculate the number of roots of $g(X)$ strictly outside the unit circle.\label{state1-2}
\State  If $g(X)$ has only one root strictly outside the unit circle, then return 1; otherwise, return 0. \label{state1-3}
\end{algorithmic}
\end{algorithm}

\begin{algorithm}
\caption{Searching the largest positive real root}
\label{alg:Sturm}
\begin{algorithmic}[1]
\Require polynomial $f\in \Z[X]$ of degree $n\ge 2$, interval $(a,b)$ with $0<a<b$ and $f(a)f(b)\ne 0$, a real number $d>0$.
\Ensure If $f$ has roots in $(a,b)$, output 1 and a small annulus $r< |z| < R$ with $R-r\le d$ and the annulus contains the largest root in $(a,b)$; otherwise, output 0.
\State $i=0$
\State Calculate the sequence $p_0(X),\ldots,p_m(X)$ in Sturm's theorem.
\If {$f$ has roots in $(a,b)$ (using Sturm's theorem) }
  \State $i=1$
  \While {$b-a> d$} \label{condition}
     \State Put $c=\frac{a+b}{2}$
     \If {$f(c)=0$} \label{step7}
       \State $a=c-\frac{1}{2}d$ \label{step8}
     \Else
        \If {$f$ has roots in $(c,b)$ (using Sturm's theorem)}
          \State $a=c$
        \Else
          \State $b=c$
        \EndIf
     \EndIf
  \EndWhile
\EndIf
\State If $i=1$, return $1,r=a,R=b$; otherwise, return 0.
\end{algorithmic}
\end{algorithm}

\subsubsection{General case II}

Since the lower bound $d_1(f)$ in \eqref{distance2} will become very small for large $n$, Algorithm \ref{test1} will be less efficient for polynomials of higher degrees. Here, we will use the lower bound $d_2(f)$ in \eqref{distance0} to design a more efficient algorithm; see Algorithm \ref{test2}.

Now, we explain briefly Algorithm \ref{test2} step by step.

Step \ref{state2-1}: the explanation is the same as Step \ref{state1-1} of Algorithm \ref{test1}, except that we apply  Algorithm \ref{alg:Sturm} by setting $d=d_2(f)$.

Step \ref{state2-2}: if $g$ is not stable, then $f$ has roots on or strictly outside the circle $|z|=R$. Note that all the real roots of $f$ are strictly inside the circle $|z|=R$, so we can judge that $f$ is not dominant.

Step \ref{state2-3}: if $g$ is stable, we have to narrow the annulus to ensure that it contains no complex (non-real) roots.

Steps \ref{state2-4} and \ref{state2-5}: the explanations are the same as Steps \ref{state1-2} and \ref{state1-3} of Algorithm \ref{test1}.

In Algorithm \ref{test2}, we first apply the lower bound $d_2(f)$ to locate a real root with the largest modulus among all real roots. If $f$ is not dominant, then the process is likely to stop in Step \ref{state2-2}; otherwise, we will use the lower bound $d_1(f)$ to narrow the small annulus $r< |z| < R$. Clearly, Algorithm \ref{test2} is more efficient than Algorithm \ref{test1}, especially when one wants to test a large number of polynomials at the same time, like Section \ref{numerical}.

\begin{algorithm}
\caption{Efficient test of dominant polynomials}
\label{test2}
\begin{algorithmic}[1]
\Require polynomial $f(X)\in \Z[X]$ of degree $n\ge 2$.
\Ensure 1 ($f$ is dominant) or 0 ($f$ is not dominant).
\State Use Sturm's theorem to locate a real root with the largest modulus among the real roots of $f$ in a small annulus $r< |z| < R$ such that $R-r\le d_2(f)$. If no non-zero real roots exist, return 0. \label{state2-1}
\State Let $g(X)=f(RX)$, and apply the Bistritz stability criterion to test whether $g(X)$  is stable or not. If $g$ is not stable, then return 0; otherwise, execute the following steps.\label{state2-2}
\State Similar as Step \ref{state2-1}, use Sturm's theorem to narrow the annulus $r< |z| < R$, and then obtain a new annulus $r_1< |z| < R_1$ with $R_1-r_1\le \frac{1}{2}d_1(f)$, where the real root is located.\label{state2-3}
\State Let $h(X)=f(r_1X)$, and apply the Bistritz stability criterion to calculate the number of roots of $h(X)$ strictly outside the unit circle.\label{state2-4}
\State  If $h(X)$ has only one root strictly outside the unit circle, then return 1; otherwise, return 0. \label{state2-5}
\end{algorithmic}
\end{algorithm}

\subsubsection{Irreducible polynomials}

Algorithm \ref{test2} is applicable to all integer polynomials of degree at least 2. However, for irreducible integer polynomials, we can make a slight improvement of Algorithm \ref{test2}; see Algorithm \ref{test3} below.

Comparing with Algorithm \ref{test2}, we only need to explain the following steps.

Step \ref{state3-1}: we gather all the exponents of $X$ in $f$ whose coefficients are non-zero, then compute their greatest common divisor. If it is greater than 1, then $f$ has such a form, and thus $f$ is not dominant.

Step \ref{state3-4}: in fact, we change the condition of Step 5 in Algorithm \ref{alg:Sturm} to be ``$b-a\ge d_1(f)$". Because here we don't need to use Lemma \ref{singular1}.

Step \ref{state3-6}:  if $h(X)$ is stable, then all the roots of $f$ are strictly inside the circle $|z|=R_1$. Note that $f$ has a real root, say $\al$, lying in the annulus $r_1<|z|<R_1$. If there exists another root lying in this annulus, then it must have modulus $|\al|$ by the choice of $r_1$ and $R_1$. Thus, by Lemma \ref{Ferguson}, $f$ is a polynomial with respect to $X^m$ for some integer $m\ge 2$; but this has been excluded in Step \ref{state3-1}. It follows that $\al$ is the dominant root of $f$, and so $f$ is dominant.

In addition, if $n\ge 3$, by the discussion below Lemma \ref{distance0}, we can let $d_1(f)$ be the following
$$
 2^{-n(n-1)(n-2)/2} (n+1)^{-(n-1)(n-2)-1/2}H(f)^{-2(n-1)(n-2)-1}.
$$

We also would like to indicate that for an irreducible polynomial $f\in \Z[X]$, $f$ has no roots in rational numbers, so we can drop Steps \ref{step7} and \ref{step8} in Algorithm \ref{alg:Sturm}.

\begin{algorithm}
\caption{Test of dominant irreducible polynomials}
\label{test3}
\begin{algorithmic}[1]
\Require irreducible polynomial $f\in \Z[X]$ of degree $n\ge 2$.
\Ensure 1 ($f$ is dominant) or 0 ($f$ is not dominant).
\State Check whether $f$ is a polynomial with respect to $X^m$ for some integer $m\ge 2$. If yes, return 0; otherwise, executive the following steps. \label{state3-1}
\State Use Sturm's theorem to locate a real root with the largest modulus among the real roots of $f$ on a small annulus $r< |z| < R$ such that $R-r\le d_2(f)$.
If no non-zero real roots exist, return 0. \label{state3-2}
\State Let $g(X)=f(RX)$, and apply the Bistritz stability criterion to test whether $g(X)$  is stable or not. If $g$ is not stable, then return 0; otherwise, execute the following steps.\label{state3-3}
\State Similar as Step \ref{state3-2}, use Sturm's theorem to narrow the annulus $r< |z| < R$, and then obtain a new annulus $r_1< |z| < R_1$ with $R_1-r_1< d_1(f)$, where the real root is located.\label{state3-4}
\State Let $h(X)=f(R_1X)$, and apply the Bistritz stability criterion to  test whether $h(X)$  is stable or not. \label{state3-5}
\State  If $h(X)$ is stable, then return 1; otherwise, return 0. \label{state3-6}
\end{algorithmic}
\end{algorithm}

\subsubsection{Complexity}\label{complexity}
Finally, we want to estimate the time complexity of Algorithms \ref{test1}, \ref{test2} and \ref{test3}.

Here, the time complexity means the total number of required arithmetic operations
(multiplication and addition). We omit the running time for comparisons and decisions, because this indeed can be ignored compared to the time required for arithmetic operations.

In fact, we only need to count the running time which is devoted to applying Sturm's theorem and the Bistritz stability criterion. According to \cite[Section 4]{Bistritz1986}, testing the Bistritz stability criterion needs $O(n^2)$ additions and multiplications.

For applying Sturm's theorem to a given polynomial $f$ of degree $n$,
we first need to compute the sequence $p_0(X),p_1(X),\ldots, p_m(X)$
by using the polynomial long division. Note that for two polynomials $g(X),h(X)\in \Z[X]$ with $\deg g \ge \deg h$, the polynomial long
division of $g$ by $h$ needs at most  $2\deg h(\deg g-\deg h+1)$ arithmetic operations. So, computing the above sequence for $f$ requires $O(n^2)$ arithmetic operations.

For a given $f(X)\in \Z[X]$ with $\deg f=n$, we use Horner's method (see \cite[Section 1.2.2]{Mignotte}) to compute the evaluation of $f$ at a point $X=a$. This requires $2n$ arithmetic operations. Thus, one call of Sturm's theorem needs $O(n^2)$ arithmetic operations. In view of Algorithm \ref{alg:Sturm} and the definitions of $d_1(f), C_1(f)$ and $C_2(f)$, there are about $n^3+n^2\log H(f)$ (up to a constant time) calls of Sturm's theorem.

Therefore, most of the running time is used on applying repeatedly Sturm's theorem, and the time complexity of Algorithm \ref{test1} is
 $O(n^5+n^4\log H(f))$.

However, if $f$ is not dominant, then the time complexity to test $f$ by using Algorithm \ref{test2} or Algorithm \ref{test3} is likely to be $O(n^3\log n+n^3\log H(f))$. In general, the time complexity of Algorithms \ref{test2} and \ref{test3} is also $O(n^5+n^4\log H(f))$.

If we know that $f$ has only real roots beforehand, then in Algorithm \ref{test2}, we can drop Step \ref{state2-3} and let $h(X)=f(rX)$ in Step \ref{state2-4}. In this case, the complexity will be $O(n^3\log n+n^3\log H(f))$.

\subsubsection{Remark}

According to Section \ref{complexity}, most of the running time is devoted to locate the real root with the largest modulus among all the real roots of the given polynomial. In view of Algorithms \ref{test1} and \ref{test2}, our strategy is to apply Sturm's theorem. Although our method is very simple, it might be not enough efficient. For example, one may adopt some root isolation methods based on Descartes' rule of signs; see \cite{Mehlhorn, Sagraloff} and the references therein.

\subsection{Numerical results}\label{numerical}

In this section, we will present some numerical results concerning dominant integer polynomials. These have been obtained by realizing the algorithms described in Section \ref{algorithm} by using PARI/GP \cite{Pari}. Here, we want to indicate that because of the limited computation resources, we haven't made computations for polynomials of higher degrees.

In order to realize the algorithms successfully in PARI/GP, we should guarantee that all the polynomials arising in the process have rational coefficients. So, we replace the quantity $d_1(f)$ by
$$2^{1-n(n-1)(n-2)/2} (n+1)^{-n(n-1)-1}H(f)^{-2n(n-1)-1} $$
or
$$2^{1-n(n-1)(n-2)/2} (n+1)^{-(n-1)(n-2)-1}H(f)^{-2(n-1)(n-2)-1} ,$$
according to which case we consider, the former one is for the general case, and the latter one is for the case of irreducible polynomials and $n\ge 3$. We also replace $d_2(f)$ with
$$
3H(f)^{1-n} (n+1)^{-n-1}.
$$

To speed up the computations, we apply Algorithm \ref{test2} if $f$ is reducible, and otherwise we apply Algorithm \ref{test3}.

Moreover, one can reduce the computations by using the fact that if $f(X)\in \Z[X]$ is dominant, then $-f(X), f(-X)$ and $-f(-X)$ are also dominant. For example, if $a_0\in \Z$ is fixed, then we have
\begin{align*}
& |\{a_0X^n+a_1X^{n-1}+\cdots +a_n \in S_n^*(H) \}|\\
& =|\{-a_0X^n+a_1X^{n-1}+\cdots +a_n \in S_n^*(H)\}|
\end{align*}
and
\begin{align*}
& |\{a_0X^n+a_1X^{n-1}+\cdots +a_n \in S_n^*(H) : a_1>0\}|\\
& =|\{a_0X^n+a_1X^{n-1}+\cdots +a_n \in S_n^*(H) : a_1<0\}|.
\end{align*}

In order to demonstrate that Algorithms \ref{test2} and \ref{test3} and their realizations are correct, we first perform some computations concerning quadratic integer polynomials and then compare the numerical results with Theorems \ref{Dn} and \ref{D2*}. One can also test quadratic integer polynomials by checking their discriminants and coefficients.

Besides, numerical results about cubic integer polynomials also can reflect the correctness of Algorithms \ref{test2} and \ref{test3} and their realizations.

For integers $n\ge 2$ and $H\ge 1$, let $M_n(H)$ be the proportion of dominant monic  integer polynomials of degree $n$ and height at most $H$ among all the monic integer polynomials of degree $n$ and height at most $H$, that is
$$
M_n(H)=\frac{D_n(H)}{(2H+1)^n}.
$$
By Theorem \ref{Dn}, we have
$$
\lim_{H\to \infty}M_n(H)=\lim_{H\to \infty}D_n(H)/(2H)^n=1.
$$

Similarly, we define
$$
P_n(H)=\frac{D_n^*(H)}{2H(2H+1)^n},
$$
and
$$
Q_n(H)=\frac{|\{f\in S_n^*(H): \textrm{$f$ is irreducible}\}|}{2H(2H+1)^n}.
$$
By Theorem \ref{D2*}, we know that
$$
\lim_{H\to \infty}P_2(H)=\lim_{H\to \infty}D_2^*(H)/(2H)^3=\frac{41+6\log 2}{72}\approx 0.6272.
$$
Theorem \ref{Dn*} implies the rate of growth of the quantity $D_n^{*}(H)$:
$$
H^{n+1}\ll D_n^{*}(H)\ll H^{n+1}.
$$
Noticing further the distribution of reducible polynomials \cite[Theorem 4]{Kuba}, we find that
$$
\lim_{H\to \infty}\frac{Q_n(H)}{P_n(H)}=1,
$$
which implies that
$$
\lim_{H\to \infty}Q_2(H)=\lim_{H\to \infty}D_2^*(H)/(2H)^3=\frac{41+6\log 2}{72}\approx 0.6272.
$$

Table \ref{MPQ2} gives the values of $M_2(H),P_2(H),Q_2(H)$ for various $H$, and it is highly consistent with the above limits. The table also suggests that the above limits can almost be achieved for small $H$, which also can be seen from Table \ref{MPQ3}. This means that to investigate statistical properties of dominant integer polynomials, we might not need to consider polynomials of large height.

Based on Table \ref{P234}, we can conjecture the following

\begin{conjecture}\label{conjecture1}
For any integer $n\ge 2$, $P_n(H)$ is an increasing function with respect to $H$.
\end{conjecture}
This conjecture implies that the limit $\lim_{H\to \infty}P_n(H)$ exists for $n\ge 3$.

According to the numerical results (especially Table \ref{Pn}), we also can make the following conjecture, which roughly says that at least half of integer polynomials are dominant.

\begin{conjecture}\label{conjecture2}
For any integer $n\ge 2$, we have $\limsup_{H\to \infty} P_n(H)>1/2$.
\end{conjecture}

\begin{table}
\centering
\caption{Values of $M_2(H),P_2(H),Q_2(H)$ for various $H$}
\label{MPQ2}
\begin{tabular}{|c|c|c|c|c|c|c|c|}
\hline
$H$ & 10 & 30 & 50 & 70 & 90 & 110 & 130\\ \hline

$M_2(H)$ & 0.7664 & 0.8707  &  0.9009 & 0.9169 & 0.9271 & 0.9343 & 0.9397\\ \hline

$P_2(H)$ & 0.5923 & 0.6148 & 0.6195 & 0.6216 & 0.6228 & 0.6236 & 0.6241\\ \hline

$Q_2(H)$ & 0.4508 & 0.5454  & 0.5722  & 0.5849 & 0.5926 & 0.5979 & 0.6016\\ \hline

\end{tabular}
\end{table}

\begin{table}
\centering
\caption{Values of $M_3(H),P_3(H),Q_3(H)$ for various $H$}
\label{MPQ3}
\begin{tabular}{|c|c|c|c|c|c|c|c|}
\hline
$H$ & 10 & 20 & 30 & 40 & 50 & 60 & 70\\ \hline

$M_3(H)$ & 0.7852 & 0.8502  &  0.8779 & 0.8944 & 0.9056 & 0.9139 & 0.9203\\ \hline

$P_3(H)$ & 0.5881 & 0.5993 & 0.6026 & 0.6043 & 0.6053 & 0.6059 & 0.6063\\ \hline

$Q_3(H)$ & 0.4962 & 0.5453  & 0.5640  & 0.5743 & 0.5807 & 0.5850 & 0.5883\\ \hline

\end{tabular}
\end{table}

\begin{table}
\centering
\caption{Values of $P_n(H)$ for $n=2,3,4$ and various $H$}
\label{P234}
\begin{tabular}{|c|c|c|c|c|c|c|c|}
\hline
$H$ & 11 & 12 & 13 & 21 & 22 & 31 & 32   \\ \hline

$P_2(H)$ & 0.5956 & 0.5978  &  0.5998 & 0.6099 & 0.6106 & 0.6152 & 0.6155  \\ \hline

$P_3(H)$ & 0.5904 & 0.5919  &  0.5935 & 0.5997 & 0.6002 & 0.6029 & 0.6031  \\ \hline

$P_4(H)$ & 0.5363 & 0.5376  &  0.5388 & 0.5443 & 0.5447 & 0.5472 & 0.5474 \\ \hline

\end{tabular}
\end{table}

\begin{table}
\centering
\caption{Values of $P_n(H)$ for $n=4,5,6$ and $H=5,10$}
\label{Pn}
\begin{tabular}{|c|c|c|c|c|}
\hline
$H$ & $P_4(H)$ & $P_5(H)$ & $P_6(H)$ \\ \hline

$5$ & 0.5155 & 0.5107  &  0.4947 \\ \hline

$10$ & 0.5345 & 0.5272  &  0.5111 \\ \hline


\end{tabular}
\end{table}

\section*{Acknowledgements}
The authors would like to thank Igor E. Shparlinski for introducing them into this topic, and also for his valuable comments on an early version of this paper. M.~S. also wants to thank the Katana team of UNSW and Xin Zhang for assisting him in computations. The research of A.~D. was supported by the Research Council of Lithuania Grant MIP-068/2013/LSS-110000-740. The research of M.~S. was supported by the Australian Research Council Grant DP130100237.

\end{document}